\documentclass[]{amsart}

\usepackage{geometry}
\usepackage{amsmath,amssymb}
\usepackage[all,poly,knot]{xy}
\usepackage{enumerate}
\usepackage{subfigure}
\usepackage{picins}

\newtheorem{Theorem}{Theorem}[section]
\newtheorem{Lemma}[Theorem]{Lemma}
\newtheorem{Proposition}[Theorem]{Proposition}
\newtheorem{Corollary}[Theorem]{Corollary}

\begin{document}
\title[Symplectic inner product graphs and their automorphisms]{Symplectic inner product graphs and their automorphisms}

\author[Zhang]{Hengbin Zhang}
\author[Zhao]{Shouxiang Zhao$^{*}$}
\author[Nan]{Jizhu Nan}
\author[Tang]{Gaohua Tang}

\thanks{*Corresponding Author}

\address[]{College of Mathematical Science, Yangzhou University, Yangzhou 225002, P.R. China} \email{hengbinzhang@yzu.edu.cn}

\address[]{School of Mathematical Sciences, Dalian University of Technology, Dalian 116024, P.R. China, and Department of Mathematics and Computer Science, Guilin Normal College, Guilin 541001, P.R. China} \email{shouxiangzhao@163.com}

\address[]{School of Mathematical Sciences, Dalian University of Technology, Dalian 116024, P.R China} \email{jznan@163.com}

\address[]{School of Science, Beibu Gulf University, 535011, P.R. China} \email{tanggaohua@163.com}


\begin{abstract} A new graph, called the symplectic inner product graph $Spi\big(2\nu,q\big)$, over a finite field $\mathbb{F}_q$ is introduced. We show that $Spi\big(2\nu,q\big)$ is connected with diameter $4$ if and only if $\nu\geq2$ and the automorphism group of $Spi\big(2\nu,q\big)$ is determined. Two necessary and sufficient conditions for two vertices and two edges of $Spi\big(2\nu,q\big)$ respectively are in the same orbit under the action of the automorphism group of $Spi\big(2\nu,q\big)$ are obtained.
\end{abstract}

\maketitle


\section{Introduction}

Let $\mathbb{F}_q$ be a finite field of order $q$ and $\nu\in \mathbb{N}$. Let $\mathbb{F}_q^{(2\nu)}= \big\{(a_1,\ldots,a_{2\nu}):a_i\in \mathbb{F}_q\ {\rm for}\ 1\leq i\leq 2\nu\big\}$ be the $2\nu$-dimensional vector space over $\mathbb{F}_q$. Let $P$ be any $m$-dimensional vector subspace of $\mathbb{F}_q^{(2\nu)}$. We will denote any $m\times 2\nu$ matrix over $\mathbb{F}_q$ whose rows form a basis of the subspace $P$ by the same letter $P$ and call the matrix $P$ a \emph{matrix representation} of the subspace $P$. For simplicity of notation, we write $[\alpha_1,\ldots,\alpha_m]$ instead of the matrix ${^t(}{^t\!\alpha}_1,\ldots,{^t\!\alpha}_m)$ for $\alpha_1,\ldots,\alpha_m\in\mathbb{F}_q^{(2\nu)}$. Let
$$K_{2\nu}=
\left(\begin{array}{ccc}
0 & I^{(\nu)}  \\
-I^{(\nu)}& 0
\end{array}\right).$$
A $2\nu\times 2\nu$ matrix $T$ over $\mathbb{F}_q$ is called a  \emph{symplectic matrix} with respect to $K_{2\nu}$ if $TK_{2\nu}{^{t}\!T}=k_{2\nu}$. Of course, $2\nu\times2\nu$ symplectic matrices are nonsingular and  they form a group under matrix multiplication, called the \emph{symplectic group} of degree $2\nu$ with respect to $K_{2\nu}$ over $\mathbb{F}_q$ and denoted by $Sp_{2\nu}(\mathbb{F}_q,K_{2\nu})$. There is an action of $Sp_{2\nu}(\mathbb{F}_q,K_{2\nu})$ on $\mathbb{F}_q^{(2\nu)}$ defined as follows:
\begin{eqnarray*}
\mathbb{F}_q^{(2\nu)}\times Sp_{2\nu}(\mathbb{F}_q,k_{2\nu}) &\to& \mathbb{F}_q^{(2\nu)}\\
((x_1,\ldots,x_{2\nu}),T) &\mapsto& (x_1,\ldots,x_{2\nu})T.
\end{eqnarray*}
The vector space $\mathbb{F}_q^{(2\nu)}$ together with this group action is called the $2\nu$-dimensional \emph{symplectic space} over $\mathbb{F}_q$. The $m$-dimensional vector subspace $P$ of $\mathbb{F}_q^{(2\nu)}$ is called a \emph{symplectic subspace of type} $(m,s)$ with respect to $K_{2\nu}$ if the rank of the matrix $PK_{2\nu}{^{t}\!P}$ is the non-negative integer $2s$. The map: $\mathbb{F}_q^{(2\nu)}\times \mathbb{F}_q^{(2\nu)}\to \mathbb{F}_q$ given by $(\alpha,\beta)\mapsto\alpha K_{2\nu}{^t\!\beta}$ is called a \emph{symplectic inner product} of $\mathbb{F}_q^{(2\nu)}$ with respect to $K_{2\nu}$. For any subspace $P$ of $\mathbb{F}_q^{(2\nu)}$, the set $\{\alpha\in\mathbb{F}_q^{(2\nu)}:\alpha K_{2\nu}{^t\!\beta}=0{\rm\ for\ all\ } \beta\in P\},$ denoted $P^\perp,$ is called the \emph{dual subspace} of $P$ with respect to $K_{2\nu}$.
\medskip

A graph $G$ is \emph{connected} if there is a path between any two vertices of $G$; otherwise the graph is disconnected. The \emph{distance} between vertices $a$ and $b$ in $G$, denoted d$(a,b)$, is the number of edges in a shortest path between $a$ and $b$. The \emph{diameter} of $G$, denoted ${\rm diam}(G)$, is the greatest distance between any two vertices of $G$. The \emph{degree} of $x\in G$, denoted $\deg(x)$, is defined to be the number of edges of the form $x$---$y$ in $G$.

\medskip
In general, determining all the automorphisms of a graph is an important but difficult problem both in graph theory and in algebraic theory. In 2006, Tang and Wan \cite{ztzw} introduced the concept of the symplectic graph over a finite field and determined its automorphism group. In \cite{zgzw,zwkz,zwkz2}, the authors introduced the concepts of the orthogonal graph and the unitary graph over a finite field and determined their automorphism groups, respectively. In 2014, Wong et al.\:\cite{dwxmjz} introduced the concept of the zero-divisor graph based on rank one upper triangular matrices and also determined its automorphism group. Motivated by previous studies, we introduce a new graph on non-trivial symplectic spaces over a finite field for studying the interplay between properties of symplectic subspaces and the structure of graphs.
\medskip

The \emph{symplectic inner product graph} with respect to $K_{2\nu}$ over $\mathbb{F}_q$, denoted $Spi\big(2\nu,q\big)$, is the graph defined on all non-trivial symplectic subspaces of $\mathbb{F}_q^{(2\nu)}$ with an edge between vertices $A$ and $B$, denoted $A$---$B$, if and only if $AK_{2\nu}{^t\!B}=0$. The set of all vertices and all edges of $Spi\big(2\nu,q\big)$ are denoted by $V\big(Spi\big(2\nu,q\big)\big)$ and $E\big(Spi\big(2\nu,q\big)\big)$, respectively. Of course, if $A\text{---}B\in E\big(Spi\big(2\nu,q\big)\big)$, then $B$ is a subspace of $A^\perp$ and $A$ is also a subspace of $B^\perp$.
\medskip

In section 2, we show that $Spi\big(2\nu,q\big)$ is connected with diameter $4$ when $\nu\geq2$ and also show that $Spi\big(2\nu_1,q_1\big)\cong Spi\big(2\nu_2,q_2\big)$ if and only if $\nu_1=\nu_2$ and $q_1=q_2$. In section 3, we obtain two necessary and sufficient conditions: (1) two vertices of $V\big(Spi\big(2\nu,q\big)\big)$ are in the same orbit under the action of ${\rm Aut}\big(Spi\big(2\nu,q\big)\big)$ if and only if they are the same type subspace; (2) two edges $g$ and $f$ of $E\big(Spi\big(2\nu,q\big)\big)$ are in the same orbit under the action of ${\rm Aut}\big(Spi\big(2\nu,q\big)\big)$ if and only if the following conditions hold: (\romannumeral1) one vertex of $g$ and one vertex of $f$ are the same type subspace, (\romannumeral2) the other vertex of $g$ and the other vertex of $f$ are the same type subspace, and (\romannumeral3) the sum of the two vertices of $g$ and the sum of the two vertices of $f$ are also the same type subspace. Moreover, we show that ${\rm Aut}\big(Spi\big(2\nu,q\big)\big)= PSp_{2\nu}\big(\mathbb{F}_q,K_{2\nu}\big)\cdot E_{2\nu}$, where $E_{2\nu}$ is a subgroup of ${\rm Aut}\big(Spi\big(2\nu,q\big)\big)$.

\section{\boldmath Some properties of the symplectic inner product graph}



In $\mathbb{F}_q^{(2\nu)}$, let $e_i$ denote the vector whose $i$-th component is 1 and all other components are 0 for $1\leq i\leq 2\nu.$

\begin{Theorem}\label{22} Let $Spi\big(2\nu,q\big)$ be the symplectic inner product graph respect to $K_{2\nu}$ over $\mathbb{F}_q$. Then $Spi\big(2\nu,q\big)$ is connected if and only if $\nu\geq2$. Moreover, if\, $Spi\big(2\nu,q\big)$ is connected, then diam$\big(Spi\big(2\nu,q\big)\big)   =   4$.
\end{Theorem}

\begin{proof} Let $\nu=1$. Since $[e_1]=[e_1]^{\bot}$, $[e_2]=[e_2]^{\bot}$ and $[e_1,e_2]^{\bot}=0$, we have d$([e_1],[e_2])=\infty$ and so $Spi\big(2,q\big)$ is not connected.

Let $\nu\geq2$ and $A,B\in V\big(Spi\big(2\nu,q\big)\big)$. If $A$ and $B$ are adjacent, then $A$---$B$ is a path of length 1. Without loss of generality we can assume that $A$ and $B$ are not adjacent. Then there exist $[\alpha],[\beta]\in V\big(Spi\big(2\nu,q\big)\big)$ such that
$$[\alpha]\subseteq A^\perp\ {\rm and}\ [\beta]\subseteq B^\perp.$$
If $[\alpha]$ and $[\beta]$ are adjacent, then $A$---$[\alpha]$---$[\beta]$---$B$ is a path of length 3. If $[\alpha]$ and $[\beta]$ are not adjacent, then $\dim([\alpha,\beta]^{\perp})=2\nu-2\geq2$ and there exists $[\gamma]\subseteq [\alpha,\beta]^{\perp}$ such that $A$---$[\alpha]$---$[\gamma]$---$[\beta]$---$B$ is a path of length 4. So $Spi\big(2\nu,q\big)$ is connected and diam$\big(Spi\big(2\nu,q\big)\big)\leq4$. Let $A=[e_1,\ldots,e_{2\nu-1}]$ and $B=[e_1,\ldots,e_{\nu-1},e_{\nu+1},\ldots,e_{2\nu}]$ with $A^{\bot}=[e_\nu]$ and $B^{\bot}=[e_{2\nu}]$. Then $A$---$[e_\nu]$---$[e_1]$---$[e_{2\nu}]$---$B$ is a shortest path of length 4 between $A$ and $B$ and thus diam$\big(Spi\big(2\nu,q\big)\big)=4$.
\end{proof}

The set of all vertices of dimension one in $Spi\big(2\nu,q\big)$ can induce a subgraph of $Spi\big(2\nu,q\big)$, denoted $Spi\big(2\nu,1,q\big)$, with an edge between vertices $A$ and $B$ if and only if $AK_{2\nu}{^t\!B}=0$. Clearly, if $X\in V\big(Spi\big(2\nu,1,q\big)\big)$, then $X$---$X^{\perp}$ is in $E\big(Spi\big(2\nu,q\big)\big)$ with $\deg(X^{\perp})=1$. Let $\mathcal{M}(m,s;2\nu)$ denote the set of subspaces of type $(m,s)$ of $\mathbb{F}_q^{(2\nu)}$ with respect to $K_{2\nu}$ and write $N(m,s;2\nu)=|\mathcal{M}(m,s;2\nu)|$. 
Then $\big|V\big(Spi\big(2\nu,1,q\big)\big)\big|=N(1,0;2\nu)=\left.{(q^{2\nu}-1)}\middle/ (q-1)\right.$ by \cite[Theorem 3.18]{zwG}.

\begin{Theorem} Let $\mathbb{F}_{q_1}$ and\, $\mathbb{F}_{q_2}$ be finite fields. Then $Spi\big(2\nu_1,q_1\big)\cong Spi\big(2\nu_2,q_2\big)$ if and only if $\nu_1=\nu_2$ and $q_1=q_2.$
\end{Theorem}
\begin{proof} ($\Leftarrow$) Suppose that $\nu_1=\nu_2,$ $\delta_1=\delta_2$ and $q_1=q_2$. Then $Spi\big(2\nu_1,q_1\big)\cong Spi\big(2\nu_2,q_2\big).$
\vskip 0.13cm

($\Rightarrow$) Suppose that $Spi\big(2\nu_1,q_1\big)\cong Spi\big(2\nu_2,q_2\big)$. It is easily seen that the set of vertices with degree one is the set of subspaces with dimension $2\nu -1$. Therefore, the graph isomorphism preserves the the set of subspaces with dimension one. We consider the induce subgraph $Spi\big(2\nu_i,1,q_i\big)$ of $Spi\big(2\nu_i,q_i\big)$ for $i=1,2$. So $Spi\big(2\nu_1,1,q_1\big)\cong Spi\big(2\nu_2,1,q_2\big)$. Then $$\left.{(q_1^{2\nu_1}-1)}\middle/ (q_1-1)\right.= \left.{(q_2^{2\nu_2}-1)}\middle/ (q_2-1)\right.$$
Obviously, $M_i=[e_1,e_2\ldots,e_{\nu_i}]$ is a maximum complete subgraph of $Spi\big(2\nu_i,1,q_i\big)$ such that $XK_{2\nu_i}{^t\!Y} = 0 $ for $X,Y\in M_i$. Then $\nu_1=|M_1|=|M_2|=\nu_2$ and $q_1=q_2.$
\end{proof}




\section{The automorphism group of the symplectic inner product graph}

\begin{Lemma}\label{l1} Let $Spi\big(2\nu,q\big)$ be the symplectic inner product graph respect to $K_{2\nu}$ over $\mathbb{F}_q$ and $\sigma\in{\rm Aut}\big(Spi\big(2\nu,q\big)\big)$. Then $\sigma(\mathcal{M}(1,0;2\nu))=\mathcal{M}(1,0;2\nu)$.
\end{Lemma}

\begin{proof}
Let $X\in \mathcal{M}(1,0;2\nu)$. Then $XK_{2\nu}{^{t}\!(X^{\bot})}=\sigma(X)K_{2\nu}{^t\sigma(\!X^{\bot})}=0$ with $\deg(X^{\bot})=1$ and $\deg(\sigma(X^{\bot}))=1$. So $\dim(\sigma(X))=1$ and $\sigma(X)\in\mathcal{M}(1,0;2\nu)$. Conversely, suppose that $\alpha(X^{\bot})=Y^{\bot}$ with $\deg(Y^{\bot})=1$ for some $\alpha\in {\rm Aut}\big(Spi\big(2\nu,q\big)\big)$. So there exist $\sigma=\alpha^{-1}\in{\rm Aut}\big(Spi\big(2\nu,q\big)\big)$ and $Y\in\mathcal{M}(1,0;2\nu)$ such that  $\sigma(Y)K_{2\nu}{^t\sigma(\!Y^{\bot})}=0$ and $X=\sigma(Y)\in\sigma(\mathcal{M}(1,0;2\nu))$. Therefore, $\sigma(\mathcal{M}(1,0;2\nu))= \mathcal{M}(1,0;2\nu)$.
\end{proof}

The quotient group of $Sp_{2\nu}(\mathbb{F}_q,K_{2\nu})$ by $C_{\nu}$, denoted $PSp_{2\nu}(\mathbb{F}_q,K_{2\nu}),$ is said to be the \emph{projective symplectic group} of degree $2\nu$ with respect to $K_{2\nu}$ over $\mathbb{F}_q$, where $C_{\nu}=\{I^{(2\nu)},-I^{(2\nu)}\}$.

\begin{Proposition} \label{p2}
Let $Spi\big(2\nu,q\big)$ be the symplectic inner product graph respect to $K_{2\nu}$ over $\mathbb{F}_q$ and $T\in Sp_{2\nu}(\mathbb{F}_q,K_{2\nu})$. Define a map
\begin{eqnarray*}
\sigma_T:V\big(Spi\big(2\nu,q\big)\big) &\to& V\big(Spi\big(2\nu,q\big)\big)\\
A&\mapsto& AT.
\end{eqnarray*}
Then

${\rm(1)}$ $\sigma_T\in {\rm Aut}\big(Spi\big(2\nu,q\big)\big);$

${\rm(2)}$ for any $T_1,T_2\in Sp_{2\nu}(\mathbb{F}_q,K_{2\nu})$, $\sigma_{T_1}=\sigma_{T_2}$ if and only if\, $T_1=\pm T_2$.
\end{Proposition}

\begin{proof}
(1) Let $T\in Sp_{2\nu}(\mathbb{F}_q,K_{2\nu})$. It is easy to check that $\sigma_T$ is a bijection. For every $A,B\in V\big(Spi\big(2\nu,q\big)\big)$, since $AK_{2\nu}{^t\!B}=ATK_{2\nu}{^t\!(BT)}$,  we have $A$---$B$ if and only if $\sigma_T(A)$---$\sigma_T(B)$. So $\sigma_T\in{\rm Aut}\big(Spi\big(2\nu,q\big)\big).$
\vskip 0.1cm

(2) ($\Leftarrow$) If $T_1=\pm T_2\in Sp_{2\nu}(\mathbb{F}_q,K_{2\nu})$, then it is easy to check that $\sigma_{T_1}=\sigma_{T_2}$.

($\Rightarrow$) Suppose that $\sigma_{T_1}=\sigma_{T_2}$. Then we need only consider them acting on the induce subgraph $Spi\big(2\nu,1,q\big)$ of $Spi\big(2\nu,q\big)$. By Lemma \ref{l1}, we have $\sigma_{T_i}(Spi\big(2\nu,1,q\big)\big)=Spi\big(2\nu,1,q\big)$ for $i=1,2$. Take $[\alpha]=[e_1],[e_2],\allowbreak\ldots,[e_{2\nu}]$. So $[\alpha]T_1=[\alpha]T_2$. Then $T_1=kT_2$ for some $k\in \mathbb{F}_q^*$. Since $T_1,T_2\in Sp_{2\nu}(\mathbb{F}_q,K_{2\nu})$, we have $k=\pm1$.
\end{proof}

By Proposition \ref{p2}, every symplectic matrix of $Sp_{2\nu}(\mathbb{F}_q,K_{2\nu})$ can induce an automorphism of $Spi\big(2\nu,q\big)$ and two different symplectic matrices $T_1$ and $T_2$ induce the same automorphism of $Spi\big(2\nu,q\big)$ if and only if $T_1=\pm T_2$. Hence $PSp_{2\nu}(\mathbb{F}_q,K_{2\nu})$ can be regarded as a subgroup of ${\rm Aut}\big(Spi\big(2\nu,q\big)\big).$

Before considering the problem of the automorphism group of $Spi\big(2\nu,q\big)$, we first introduce some notation and terminology that will be used in this paper. In what follows, we write $f_i$ for $e_{\nu+i}$, $1\leq i\leq \nu$. Then
$$e_i{K_{2\nu}}^t\!e_j=f_i{K_{2\nu}}^t\!f_j=0\ {\rm for}\ 1\leq i,j\leq\nu$$
and
$$e_i{K_{2\nu}}^t\!f_j=0,\: e_i{K_{2\nu}}^t\!f_i=1\ {\rm for}\ i\neq j,\: 1\leq i,j\leq\nu.$$
Let $\varphi_{2\nu}$ be the natural action of ${\rm Aut}(\mathbb{F}_q)$ on the group $\mathbb{F}_q^*\times \cdots\times\mathbb{F}_q^*$ defined by
$$\varphi_{2\nu}(\pi)\big((k_1,\ldots,k_{\nu})\big)=(\pi(k_1),\ldots,\pi(k_{\nu}))$$
for $\pi\in{\rm Aut}(\mathbb{F}_q)$ and $k_1,\ldots,k_{\nu}\in\mathbb{F}_q^*$. The \emph{semidirect product} of $\mathbb{F}_q^*\times \cdots\times\mathbb{F}_q^*$ by ${\rm Aut}(\mathbb{F}_q)$ with respect to $\varphi_{2\nu}$, denoted $(\mathbb{F}_q^*\times\cdots\times\mathbb{F}_q^*)\rtimes_{\varphi_{2\nu}} {\rm Aut}(\mathbb{F}_q)$, is the group consisting of all elements of the form $(k_1,\ldots, k_{\nu},\pi)$, with the multiplication defined as follows
$$(k_1,\ldots,k_{\nu},\pi)(k_1^{'},\ldots,k_{\nu}^{'},\pi^{'})=(k_1\pi(k_1^{'}),\ldots, k_{\nu}\pi(k_{\nu}^{'}),\pi\pi^{'}).$$
Define maps:
\begin{eqnarray*}
\sigma_{\pi}:V\big(Spi\big(2\nu,q\big)\big) & \to & V\big(Spi\big(2\nu,q\big)\big)\\
\big(a_{ij}\big)_{m\times 2\nu}&\mapsto&\big(\pi(a_{ij})\big)_{m\times 2\nu}\\
{\rm and}\ \ \ \ \ \ \ \ \ \ \ \ \ \ \ \ \ \ \ \ \ \ \ \ \ \ \ \ \ \ \ \ \ \ \ \ \ \ \ \ \ \ \ \ \ \ \ \ \ \ \ &~&\ \ \ \ \ \ \ \ \ \ \ \ \ \ \ \ \ \ \ \ \ \ \ \ \ \ \ \ \ \  \ \ \ \ \ \ \ \ \ \ \ \ \ \ \ \ \ \ \ \ \ \ \ \ \ \ \ \ \ \ \ \ \ \ \ \ \\
\sigma_{(k_1,\ldots,k_\nu,\pi)}:V\big(Spi\big(2\nu,q\big)\big) &\to& V\big(Spi\big(2\nu,q\big)\big)\\
A &\mapsto& \sigma_\pi(A)\cdot{\rm diag}(k_1,\ldots,k_{\nu}, k_1^{-1},\ldots,k_{\nu}^{-1}),
\end{eqnarray*}
where $\pi\in{\rm Aut}(\mathbb{F}_q)$ and $k_1,\ldots,k_{\nu}\in \mathbb{F}_q^*$. It is easy to check that $\sigma_{\pi}\in {\rm Aut}\big(Spi\big(2\nu,q\big)\big)$. Moreover, since
${\rm diag}(k_1,\ldots,k_{\nu}, k_1^{-1},\ldots,k_{\nu}^{-1}) K_{2\nu}{^t({\rm diag}}(k_1,\ldots,k_{\nu},k_1^{-1},\ldots,k_{\nu}^{-1}))=-K_{2\nu},$
we have $${\rm diag}(k_1,\ldots,k_{\nu}, k_1^{-1},\ldots,k_{\nu}^{-1})\in PSp_{2\nu}\big(\mathbb{F}_q,K_{2\nu}\big).$$
By Proposition \ref{p2}, we have $\sigma_{(k_1,\ldots,k_{\nu},\pi)}\in {\rm Aut}\big(Spi\big(2\nu,q\big)\big).$

\vskip 0.1cm

Before studying the automorphism group of $Spi\big(2\nu,q\big),$ let us look at a simple example of $Spi\big(2,q\big).$

\newtheorem{theorem}{Example}[section]

\begin{theorem}\label{34}
{\rm In} $Spi\big(2,q\big),$ {\rm we have}
$$V\big(Spi\big(2,q\big)\big)=\big\{[(1,x)]:x\in\mathbb{F}_q\} \bigcup\{[(0,1)]\big\}$$
{\rm and}
$$E\big(Spi\big(2,q\big)\big)=\big\{[(x,y)]\text{---}[(x,y)]: [(x,y)]\in V\big(Spi\big(2,q\big)\big)\big\}.$$
{\rm By Theorem \ref{22}}, $Spi\big(2,q\big)$ {\rm is disconnected}.
\end{theorem}

\medskip

\begin{Theorem}\label{t3}
Let $Spi\big(2\nu,q\big)$ be the symplectic inner product graph respect to $K_{2\nu}$ over $\mathbb{F}_q$ and $\nu\geq2$. $E_{2\nu}$ is a subgroup of ${\rm Aut}\big(Spi\big(2\nu,q\big)\big)$ defined as follows
$$\big\{\sigma\in{\rm Aut}\big(Spi\big(2\nu,q\big)\big):\sigma([e_{i}])=[e_{i}]\ {\rm and}\ \sigma([f_{i}])=[f_{i}]\ {\rm for}\ 1\leq i\leq\nu\big\}.$$
Then ${\rm Aut}\big(Spi\big(2\nu,q\big)\big)=PSp_{2\nu}(\mathbb{F}_q,K_{2\nu})\cdot E_{2\nu}$. Moreover, let
\begin{eqnarray*}
h:\left.{\big((\mathbb{F}_q^*\times\cdots\times \mathbb{F}_q^*)\rtimes_{\varphi_{2\nu}}{\rm Aut}(\mathbb{F}_q)\big)}\middle/ Z_{2\nu}\right. &\to& E_{2\nu}\\
(k_1,\ldots,k_\nu,\pi)Z_{2\nu} &\mapsto& \sigma_{(k_1,\ldots,k_\nu,\pi)},
\end{eqnarray*}
where $Z_{2\nu}=\{(s,\ldots,s,1_{\mathbb{F}_q}) \in(\mathbb{F}_q^*\times\cdots\times \mathbb{F}_q^*)\rtimes_{\varphi_{2\nu}}{\rm Aut}(\mathbb{F}_q):s=\pm1\}$ and $1_{\mathbb{F}_q}$ is an identity element of ${\rm Aut}(\mathbb{F}_q).$
Then $h$ is a group isomorphism.
\end{Theorem}

\begin{proof} Let $\tau\in{\rm Aut}\big(Spi\big(2\nu,q\big)\big)$. By Lemma \ref{l1}, We consider it acting on the induced subgraph $Spi\big(2\nu,1,q\big)$ of $Spi\big(2\nu,q\big)$. By \cite[Theorem 3.4]{ztzw}, there exists $T\in PSp_{2\nu}(\mathbb{F}_q,K_{2\nu})$ such that $\sigma_T\tau\in E_{2\nu}$. Let $\tau_1=\sigma_T\tau$, we have
$$\tau=\sigma_T^{-1}\tau_1\in PSp_{2\nu}(\mathbb{F}_q,K_{2\nu})\cdot E_{2\nu}.$$
So ${\rm Aut}\big(Spi\big(2\nu,q\big)\big)=PSp_{2\nu}(\mathbb{F}_q,K_{2\nu})\cdot E_{2\nu}$. By the proof of \cite[Theorem 3.4]{ztzw}, we know that for any $\sigma\in E_{2\nu}$, there exist $k_1,\ldots,k_{\nu}\in\mathbb{F}_q^*$ and $\pi\in{\rm Aut}(\mathbb{F}_q)$ such that $\sigma=\sigma_{(k_1,\ldots,k_{\nu},\pi)}$.
So for $[(a_1,\ldots,a_\nu,a_{\nu+1},\ldots,a_{2\nu})]\in Spi\big(2\nu,1,q\big)$, we have
$$\sigma\big([(a_1,\ldots,a_\nu,a_{\nu+1},\ldots,a_{2\nu})]\big)= \big[(k_1\pi(a_1),\ldots,k_\nu\pi(a_\nu), k_1^{-1}\pi(a_{\nu+1}),\ldots, k_{\nu}^{-1}\pi(a_{2\nu}))\big].$$

For any $\big(a_{ij}\big)_{m\times2\nu}\in V\big(Spi\big(2\nu,q\big)\big),$ we have
\begin{displaymath}
\sigma\big(\big(a_{ij}\big)_{m\times2\nu}\big)=\left(
\begin{array}{cccccc}
k_1\pi(a_{11})&\cdots&k_\nu\pi(a_{1\nu}) &k_1^{-1}\pi(a_{1(\nu+1)}) &\cdots& k_\nu^{-1}\pi(a_{1(2\nu)})\\
\vdots&\ddots&\vdots &\vdots&\ddots&\vdots\\
k_1\pi(a_{m1})&\cdots&k_\nu\pi(a_{m\nu}) &k_1^{-1}\pi(a_{m(\nu+1)}) &\cdots& k_\nu^{-1}\pi(a_{m(2\nu)})
  \end{array}
\right).
\end{displaymath}
So $\sigma=\sigma_{(k_1,\ldots,k_\nu,\pi)}$ and $h$ is a surjective function. It is easy to check that $h$ is a group homomorphism and the set $Z_{2\nu}$ is the kernel of $h$. Thus $h$ is a group isomorphism.
\end{proof}

\begin{Corollary}\label{36} Let $\nu\geq 1.$ Then
\begin{displaymath}
\big|{\rm Aut}\big(Spi\big(2\nu,q\big)\big)\big|=\left\{ \begin{array}{ll}
(q+1)!, & \textrm{if $\nu=1,$}\\
\left.{q^{2\nu} \prod_{i=1}^{\nu}(q^{2i}-1)[\mathbb{F}_q:\mathbb{F}_p]}\middle/ 2,\right.& \textrm{if $\nu\geq2,$}
\end{array} \right.
\end{displaymath}
where $p={\rm char}(\mathbb{F}_q).$
\end{Corollary}

\vskip 0.15cm

\begin{Theorem}\label{32} Let $Spi\big(2\nu,q\big)$ be the symplectic inner product graph respect to $K_{2\nu}$ over $\mathbb{F}_q$ and\, $2s\leq m\leq \nu+r$. Then $\mathcal{M}(m,s;2\nu)$ is exactly one orbit of $V\big(Spi\big(2\nu,q\big)\big)$ under the action of ${\rm Aut}\big(Spi\big(2\nu,q\big)\big).$
\end{Theorem}

\begin{proof}
Let $\varphi\in{\rm Aut}\big(Spi\big(2\nu,q\big)\big)$. Then by Theorem \ref{t3}, there exist $\tau\in PSp_{2\nu}(\mathbb{F}_q,K_{2\nu})$, $k_1,\ldots,k_{\nu}\in\mathbb{F}_q^*$ and $\pi\in{\rm Aut}(\mathbb{F}_q)$ such that
$$\varphi=\tau\cdot\sigma_\pi\cdot{\rm diag}(k_1,\ldots,k_{\nu}, k_1^{-1},\ldots,k_{\nu}^{-1}).$$
Let $A\in\mathcal{M}(m,s;2\nu)$. It is easy to check that $\sigma_\pi(A)\in\mathcal{M}(m,s;2\nu)$. Since $\tau,{\rm diag}(k_1,\ldots,k_{\nu}, k_1^{-1},\ldots,\allowbreak k_{\nu}^{-1})\in PSp_{2\nu}(\mathbb{F}_q,K_{2\nu})$, we have $\tau(\sigma_\pi({\rm diag}(k_1,\ldots,k_{\nu}, k_1^{-1},\ldots,k_{\nu}^{-1})(A)))\in \mathcal{M}(m,s;2\nu)$ by \cite[Theorem 3.7]{ztzw}. Thus $\sigma(\mathcal{M}(m,s;2\nu))=\mathcal{M}(m,s;2\nu)$.

Let $A,B\in\mathcal{M}(m,s;2\nu)$. Then by \cite[Theorem 3.7]{ztzw} and Proposition \ref{p2}, there exists $T\in Sp_{2\nu}(\mathbb{F}_q,K_{2\nu})$ such that
$$\sigma_T(A)=B\ {\rm and}\ \sigma_T\in{\rm Aut}\big(Spi\big(2\nu,q\big)\big).$$
Thus $\mathcal{M}(m,s;2\nu)$ is exactly one orbit of $V\big(Spi\big(2\nu,q\big)\big)$ under the action of ${\rm Aut}\big(Spi\big(2\nu,q\big)\big).$
\end{proof}


Let $t$ be a map from $V\big(Spi\big(2\nu,q\big)\big)$ to $\mathbb{N}\times\mathbb{N}$ given by
\begin{center}
$X\mapsto t\big(X\big)=\big(m,s\big)$,
\end{center}
where $X$ is of type $\big(m,s\big).$

\begin{Lemma}\label{l6} Let $Spi\big(2\nu,q\big)$ be the symplectic inner product graph respect to $K_{2\nu}$ over $\mathbb{F}_q$ and $X_1\text{---}X_2,Y_1\text{---}Y_2\in E\big(Spi\big(2\nu,q\big)\big).$ Then $X_1$---$X_2$ and $Y_1$---$Y_2$ are in the same orbit of $E\big(Spi\big(2\nu,q\big)\big)$ under the action of $Sp_{2\nu}\big(\mathbb{F}_q,K_{2\nu}\big)$ if and only if one of the following is true$:$ $(1)$ $t(X_1)=t(Y_1),$ $t(X_2)=t(Y_2),$ $t(X_1+X_2)=t(Y_1+Y_2);$ $(2)$ $t(X_1)=t(Y_2),$ $t(X_2)=t(Y_1),$ $t(X_1+X_2)=t(Y_1+Y_2).$
\end{Lemma}

\begin{proof}
($\Rightarrow$) Suppose that $X_1$---$X_2$ and $Y_1$---$Y_2$ are in the same orbit of $E\big(Spi\big(2\nu,q\big)\big)$ under the action of $Sp_{2\nu}\big(\mathbb{F}_q,K_{2\nu}\big)$.
Then there exists $T\in Sp_{2\nu}\big(\mathbb{F}_q,K_{2\nu}\big)$ such that one of the following is true: $X_1T=Y_1,$ $X_2T=Y_2$; $X_1T=Y_2,$ $X_2T=Y_1.$ Without loss of generality we can assume that $X_1T=Y_1$ and $X_2T=Y_2$. Then $t(X_1)=t(Y_1)$ and $t(X_2)=t(Y_2)$ by Theorem \ref{32}. It is sufficient to prove $t(X_1+X_2)=t(Y_1+Y_2)$. Let $t(X_1)=t(Y_1)=(m_1,s_1)$ and $t(X_2)=t(Y_2)=(m_2,s_2)$. Then
$$X_1T=[X_{11},\ldots,X_{1m_1}]T=[Y_{11},\ldots,Y_{1m_1}]=Y_1\ $$
and
$$X_2T=[X_{21},\ldots,X_{2m_2}]T=[Y_{21},\ldots,Y_{2m_2}]=Y_2.$$
It is easy to see that there exist $1\leq t_1<\cdots <t_r\leq m_2$ such that
$$X_1+X_2=[X_{11},\ldots,X_{1m_1},X_{2t_1},\ldots,X_{2t_r}]\,\
{\rm and}\,\
Y_1+Y_2=[Y_{11},\ldots,Y_{1m_1},Y_{2t_1},\ldots,Y_{2t_r}].\, $$
So $(X_1+X_2)T=(Y_1+Y_2)$ and then $t(X_1+X_2)=t(Y_1+Y_2)$ by Proposition \ref{p2} and Theorem \ref{32}. When $X_1T=Y_2$ and $X_2T=Y_1,$ we conclude similarly that $t(X_1)=t(Y_2),$ $t(X_2)=t(Y_1)$ and $t(X_1+X_2)=t(Y_1+Y_2).$
\vskip 0.13cm

($\Leftarrow$) Suppose that one of the following is true: $(1)$ $t(X_1)=t(Y_1),$ $t(X_2)=t(Y_2),$ $t(X_1+X_2)=t(Y_1+Y_2)$; $(2)$ $t(X_1)=t(Y_2),$ $t(X_2)=t(Y_1),$ $t(X_1+X_2)=t(Y_1+Y_2)$. Without loss of generality we can assume that $t(X_1)=t(Y_1)=(m_1,s_1),$ $t(X_2)=t(Y_2)=(m_2,s_2)$ and $t(X_1 + X_2)=t(Y_1+Y_2)=(m,s).$ Since $X_1$---$X_2\in E\big(Spi\big(2\nu,q\big)\big)$, there exist matrix representations
$[\alpha_{1},\ldots,\alpha_{m-m_2},\gamma_1,\ldots,\gamma_{m_1+m_2-m}],$ $[\beta_{1},\ldots,\beta_{m-m_1},\gamma_1,\ldots,\gamma_{m_1+m_2-m}]$ and $[\gamma_1,\ldots,\gamma_{m_1+m_2-m}]$ of $X_1,$ $X_2$ and $X_1\bigcap X_2$ respectively such that
$$X_1+X_2=\big[\alpha_{1},\ldots,\alpha_{m-m_2},\beta_{1}, \ldots,\beta_{m-m_1}, \gamma_1,\ldots,\gamma_{m_1+m_2-m}\big]$$
and
$$\alpha_iK_{2\nu}{^t\!\beta_j}=0,\: \alpha_iK_{2\nu}{^t\!\gamma_k}=0,\: \beta_jK_{2\nu}{^t\!\gamma_k}=0,$$
where $1\leq i\leq m-m_2$, $1\leq j\leq m-m_1$ and $1\leq k\leq m_1+m_2-m$. It is easy to verify that $t(X_1\bigcap X_2)=(m_1+m_2-m,s_1+s_2-s).$ We can choose a suitable basis
$\alpha_{1},\ldots,\alpha_{m-m_2},$ $\beta_{1},\ldots,\beta_{m-m_1},$ $ \gamma_1,\ldots,\linebreak[4]\gamma_{m_1+m_2-m}$
of $X_1+X_2$ and a suitable basis $[\alpha_{1}^{'},\ldots,\alpha_{m-m_2}^{'},\beta_{1}^{'}, \ldots,\beta_{m-m_1}^{'}, \gamma_1^{'},\ldots,\gamma_{m_1+m_2-m}^{'}]$
of $Y_1+Y_2$ such that
$Y_1=[\alpha_{1}^{'},\ldots,\alpha_{m-m_2}^{'},\gamma_1^{'},\ldots, \gamma_{m_1+m_2-m}^{'}],$ $Y_2=[ \beta_{1}^{'},\ldots, \beta_{m-m_1}^{'},\gamma_1^{'},\ldots,\gamma_{m_1+m_2-m}^{'}]$, and

$$(X_1+X_2)K_{2\nu}{^t(X_1+X_2)}=(Y_1+Y_2)K_{2\nu}{^t(Y_1+Y_2)}$$
By \cite[Lemma 3.11]{zwG}, there exists $T\in Sp_{2\nu}\big(\mathbb{F}_q,K_{2\nu}\big)$ such that $(X_1+X_2)T=Y_1+Y_2,$ $X_1T=Y_1$ and $X_2T=Y_2$. So $X_1$---$X_2$ and $Y_1$---$Y_2$ are in the same orbit of $E\big(Spi\big(2\nu,q\big)\big)$ under the action of $Sp_{2\nu}\big(\mathbb{F}_q,K_{2\nu}\big)$ by the Proposition \ref{p2}.
\vskip 0.1cm

When $t(X_1)=t(Y_2),$ $t(X_2)=t(Y_1)$ and $t(X_1+X_2)=t(Y_1+Y_2),$ we conclude similarly that $X_1$---$X_2$ and $Y_1$---$Y_2$ are in the same orbit of $E\big(Spi\big(2\nu,q\big)\big)$ under the action of $Sp_{2\nu}\big(\mathbb{F}_q,K_{2\nu}\big)$, which completes the proof.
\end{proof}

\begin{Theorem} \label{p7} Let $Spi\big(2\nu,q\big)$ be the symplectic inner product graph respect to $K_{2\nu}$ over $\mathbb{F}_q$ and $X_1\text{---}X_2,Y_1\text{---}Y_2\in E\big(Spi\big(2\nu,q\big)\big).$ Then $X_1$---$X_2$ and $Y_1$---$Y_2$ are in the same orbit of $E\big(Spi\big(2\nu,q\big)\big)$ under the action of ${\rm Aut}\big(Spi\big(2\nu,q\big)\big)$ if and only if one of the following is true$:$ $(1)$ $t(X_1)=t(Y_1),$ $t(X_2)=t(Y_2),$ $t(X_1+X_2)=t(Y_1+Y_2);$ $(2)$ $t(X_1)=t(Y_2),$ $t(X_2)=t(Y_1),$ $t(X_1+X_2)=t(Y_1+Y_2).$
\end{Theorem}

\begin{proof} ($\Rightarrow$) Suppose that $X_1$---$X_2$ and $Y_1$---$Y_2$ are in the same orbit of $E\big(Spi\big(2\nu,q\big)\big)$ under the action of ${\rm Aut}\big(Spi\big(2\nu,q\big)\big)$. Then there exists $\sigma\in{\rm Aut}\big(Spi\big(2\nu,q\big)\big)$ such that one of the following is true: $\sigma(X_1)=Y_1$, $\sigma(X_2)=Y_2$; $\sigma(X_1)=Y_2,$ $\sigma(X_2)=Y_1$. Without loss of generality we can assume that $\sigma(X_1)=Y_1$ and $\sigma(X_2)=Y_2$. Then $t(X_1)=t(Y_1)$ and $t(X_2)=t(Y_2)$ by Theorem \ref{32}. It is sufficient to prove $t(X_1+X_2)=t(Y_1+Y_2)$. By Theorem \ref{p2}, there exist $\tau\in PSp_{2\nu}(\mathbb{F}_q,K_{2\nu})$, $k_1,\ldots,k_{\nu}\in\mathbb{F}_q^*$ and $\pi\in{\rm Aut}(\mathbb{F}_q)$ such that
$$\sigma=\tau\cdot\sigma_\pi\cdot{\rm diag}(k_1,\ldots,k_{\nu}, k_1^{-1},\ldots,k_{\nu}^{-1}).$$
It is easy to check that for any $\pi\in{\rm Aut}(\mathbb{F}_q)$, we have $\sigma_\pi(\mathcal{M}(m,s;2\nu))=\mathcal{M}(m,s;2\nu)$. By the proof of Lemma \ref{l6}, we have $t(X_1+X_2)=t(Y_1+Y_2)$.
\vskip 0.13cm

($\Leftarrow$) Suppose that one of the following is true: $(1)$ $t(X_1)=t(Y_1),$ $t(X_2)=t(Y_2),$ $t(X_1+X_2)=t(Y_1+Y_2)$; $(2)$ $t(X_1)=t(Y_2),$ $t(X_2)=t(Y_1),$ $t(X_1+X_2)=t(Y_1+Y_2).$ Then it is easy to check that $X_1$---$X_2$ and $Y_1$---$Y_2$ are in the same orbit of $E\big(Spi\big(2\nu,q\big)\big)$ under the action of ${\rm Aut}\big(Spi\big(2\nu,q\big)\big)$ by Proposition \ref{p2} and Lemma \ref{l6}, which completes the proof.
\end{proof}

\section*{Acknowledgements}

This work was supported by the National Natural Science Foundation of China (Grant No. 12171194, 11961050) and the Guangxi Natural Science Foundation (Grant No. 2021GXNSFAA220043, 2020GXNSFAA159053)

\end{document}